\documentclass[11pt,letterpaper,twoside]{amsart}
\usepackage{a4wide,charter, amsmath,amsbsy,amsfonts,amssymb,
stmaryrd,amsthm,mathrsfs,graphicx, amscd, tikz-cd, cancel}
\usepackage[normalem]{ulem}
\usetikzlibrary{matrix,arrows,decorations.pathmorphing}
\usepackage[active]{srcltx}
\usepackage[
	hypertexnames=false,
	hyperindex,
	pagebackref,
	pdftex,
	breaklinks=true,
	bookmarks=false,
	colorlinks,
	linkcolor=blue,
	citecolor=red,
	urlcolor=red,
]{hyperref}
\usepackage[mathscr]{eucal}

\newcommand\la{\langle}
\newcommand\ra{\rangle}
\newcommand\eb{\mathbf{e}}
\newcommand\fb{\mathbf{f}}

\newcommand\ZZ{{\mathbb Z}}
\newcommand\FF{{\mathbb F}}

\newcommand\QQ{{\mathbb Q}}

\newcommand\RR{{\mathbb R}}
\newcommand\CC{{\mathbb C}}

\newcommand\KK{{\mathbb K}}

\newcommand\cR{{\mathscr R}}
\newcommand\SL{\operatorname{SL}}

\newcommand\Sp{\operatorname{Sp}}

\newcommand\aut{\operatorname{Aut}}

\renewcommand\dim{\operatorname{dim}}
\newcommand\diff{\operatorname{Diff}}

\newcommand\ssm{\smallsetminus}

\renewcommand{\hom}{\operatorname{Hom}}

\newcommand\Hcal{\mathcal{H}}

\newcommand\Bs{{\mathscr B}}

\newcommand\Gs{\mathscr{G}}

\newcommand\G{\Gamma}

\newcommand\cG{{\mathscr G}}

\newcommand\GL{\operatorname{GL}}
\newcommand\Iso{\operatorname{Iso}}

\newcommand\Hl{\operatorname{H}}

\newcommand\Res{\operatorname{Res}}

\newcommand\End{\operatorname{End}}

\newcommand\EU{\operatorname{EU}}

\newcommand\SU{\operatorname{SU}}
\newcommand\U{\operatorname{U}}

\newcommand\Orth{\operatorname{O}}

\newcommand\tr{\operatorname{tr}}

\newcommand\ru{\operatorname{R_u}}
\newcommand\rank{\operatorname{rk}}

\newtheorem{theorem}{Theorem}[section]
\newtheorem{corollary}[theorem]{Corollary}
\newtheorem{proposition}[theorem]{Proposition}

\newtheorem{lemma}[theorem]{Lemma}

\theoremstyle{definition}

\newtheorem{remark}[theorem]{Remark}

\theoremstyle{empty}

\begin{document}
\title{Arithmetic  representations of mapping class groups}

\author{Eduard Looijenga}\thanks{Part of the research for this paper was done when the author was supported by the Chinese National Science Foundation and by the Jump Trading Mathlab Research Fund}
\address{Mathematisch Instituut, Universiteit Utrecht (Nederland) and Mathematics Department, University of Chicago (USA)}
\email{e.j.n.looijenga@uu.nl}


\begin{abstract}
Let $S$ be a closed oriented surface and $G$ a finite group of orientation preserving automorphisms of $S$ whose orbit space has genus at least $2$. 
There is a  natural group homomorphism  from the $G$-centralizer in $\diff^+(S)$ to the  $G$-centralizer in $\Sp(H_1(S))$. We give a sufficient condition for its image to be a subgroup of finite index.
\end{abstract}

\maketitle
\section{Introduction and statement of the main result}

Let $S$ be a closed connected oriented surface of genus $\ge 2$. The group $\diff^+(S)$ of orientation preserving diffeomorphisms of $S$ acts on 
$\Hl_1(S)$  via its connected component group $\pi_0(\diff^+(S))$, known as the \emph{mapping class group} of $S$, and it is a classical fact 
that the image of this representation  is the full symplectic group  $\Sp(\Hl_1(S))$ of integral linear transformations which preserve the intersection form on $\Hl_1(S)$. 
This paper concerns an equivariant version, where it is assumed that we are given a finite subgroup $G\subset\diff^+(S)$. 
The centralizer $\diff^+(S)^G$ 
of $G$ in $\diff^+(S)$  lands under the above symplectic representation in $\Sp(\Hl_1(S))^G$ and the question we 
address here is how much smaller the image is. Besides its intrinsic interest, the answer has consequences for understanding the 
mapping class group of  the $G$-orbit space of $S$. We shall regard the latter as an orbifold surface and denote it by $S_G$;
the regular orbits then define an open subset $S_G^\circ\subset S_G$  with finite complement. This punctured surface $S_G^\circ$ has negative  Euler characteristic. The  image  of  $\diff^+(S)^G$ in the mapping class group of the punctured surface  $S_G^\circ$ is of finite index and  thus makes $\Sp(\Hl_1(S))$ a `virtual representation' of that mapping class group.

The work of  Putman-Wieland \cite{PW} relates our question to the Ivanov conjecture as follows. Let us say
that the $G$-action on $S$ has the \emph{Putman-Wieland property} if  $\diff^+(S)^G$ has  no  finite orbit in $\Hl_1(S)\ssm\{0\}$.  These authors prove that if that property holds for a given genus $h$ of 
$S_G$  (no matter what $S$ and $G$ are), then every finite index subgroup of a mapping class group of a connected  oriented  surface of finite type of genus $>h$ has zero first Betti number. The first part of our main result is about that property.
\begin{theorem}\label{thm:main}
 Let $S\to S_G$ be a $G$-cover as above. 
\begin{enumerate}
\item[(i)] If this cover is trivial  over a compact genus one subsurface of $S_G^\circ$ with connected boundary, then the action of $\diff^+(S)^G$ on $\Hl_1(S)$ has no nonzero finite orbits. 
\item[(ii)] If this cover is trivial  over a compact genus two subsurface of $S_G^\circ$
with connected boundary, then the image of $\diff^+(S)^G$ in  $\Sp(\Hl_1(S))^G$ is of finite index.  
\end{enumerate}
\end{theorem}

We will also obtain an arithmeticity property in the setting of (i). See the discussion  and the end of this introduction as well as Remark \ref{rem:exceptions}.

\begin{remark}\label{rem:}
Note that in either case the cover over the complement of such a subsurface of $S_G$ must be connected 
(for the subsurface has a connected boundary and $S$ is connected).  Since this complement has a boundary component over which the covering is trivial, we may contract that component (and each of the components lying over it)  to obtain a  $G$-covering  $S'\to S'_G$, where the genus of $S'_G$ is now 1 resp.\  2 less than that of $S_G$.  As this covering represents all the topological input, we may paraphrase our main theorem as  saying 
that the Putman-Wieland property resp.\ the arithmeticity property holds after a `stabilization' by taking a  connected sum of the base orbifold with a closed surface of  genus 1 resp.\ 2.
\end{remark}

We will prove this theorem under the 
apparently weaker assumption  that there exists  a closed  one-dimensional 
submanifold nonempty $A\subset S_G^\circ$, so a disjoint union of say $k\ge 1 $ embedded circles (with $k=1$ in case (i) and $k=2$ in case (ii))  such  that  $S\to S_G$ is trivial over $A$ and connected over 
$S_G\ssm A$.  This looks as if this is a more general result, because it is easy to find in the respective cases such an $A$ inside  the postulated subsurface with the property that its complement  is connected.  But we will see that this generalization is only  apparent.

There is also a useful geometric interpretation for this last formulation:  given such an $A$, then we can obtain the 
$G$-covering  $S'\to S'_G$ as above as follows: regard $S_G\ssm A$ as a punctured surface (so with two punctures for each component of $A$) and let   $S'_G\supset S_G\ssm A$  be the closed orbifold obtained by filling in these punctures as non-orbifold points. Our assumptions say that $S'_G$ is a closed connected surface  (the genus drop is the number of connected components of $A$) and that  the given $G$-covering  $S\to S_G$ arises from a $G$-covering $S'\to S'_G$ with  for each component of $A$ an 
identification of the fibers of this covering over the two associated points (as principal $G$-sets). If we give $S_G$ a complex structure and 
thus turn it into a smooth complex-projective curve with an orbifold structure, then an algebraic geometer might be tempted to regard  this orbifold curve as being in its moduli space near the Deligne-Mumford stratum where the orbifold acquires $k$ nodes, but for which the $G$-covering stays irreducible and does not ramify over the nodes. The covering  $S'\to S'_G$ then appears as the normalization of such a degeneration. No algebraic geometry is used in the proof, though, for the topological part of this paper uses methods that directly generalize those of \cite{L}.

Let us compare the above theorem with  the work of Gr\"unewald-Larsen-Lubotzky-Malestein \cite{GLLM},
whose main motivation was to construct, via the virtual isomorphism mentioned above, new arithmetic quotients of the mapping class 
group of $S_G$. They assume that $G$ acts freely so that $S_G=S^\circ_G$ and  impose another, more technical condition,  which in our set-up translates into that we are  in the context of (i)  and demand  that the covering $S'\to S'_G$ is of `handlebody type', in the sense that it extends to a handlebody that has $S'_G$ as boundary.
They prove  that  the image of $\diff^+(S)^G$ in each simple factor of  $\Sp(\Hl^1(S, \QQ))^G$ is arithmetic.  
Our  approach differs from theirs in several aspects, but mostly in  our direct  and relatively  simple way of constructing  $G$-equivariant mapping classes.
\\

When speaking  of arithmetic subgroups of  $\Sp(\Hl^1(S, \QQ))^G$, it is of course tacitly understood that  the latter can be regarded as the group of rational points of an algebraic group defined over $\QQ$. Let us make this explicit.

Denote by  $X(\QQ G)$  the set of irreducible characters of $\QQ G$ and  choose for every $\chi\in X(\QQ G)$ a representing irreducible (left) $\QQ G$-module
$V_\chi$. We also fix on every $V_\chi$ a $G$-invariant inner  product $s_\chi :V_\chi\times V_\chi\to \QQ$ (which can be obtained as the $G$-average of an arbitrary inner product). This exhibits $V_\chi$ as a self-dual $\QQ G$-module.
Then $\End_{\QQ G}(V_\chi)$ is a skew field which is of finite $\QQ$-dimension. We denote its opposite by $D_\chi$ (meaning that the underling $\QQ$-vector space is $\End_{\QQ G}(V_\chi)$, but that composition is taken in opposite order) so that $V_\chi$ is now a right $D_\chi$-module. Adopting as a  convention that $D_\chi$ used as superscript (resp.\ subscript) indicates that we  are dealing with right (resp.\ left) $D_\chi$-module endomorphisms, then the natural map 
\[
\textstyle \QQ G\cong \prod_{\chi\in X(\QQ G)} \End^{D_\chi}(V_\chi),
\]
is an  isomorphism of $\QQ$-algebras. This is  in fact the Wedderburn decomposition of $\QQ G$, as each factor is  a minimal $2$-sided ideal.

The group algebra $\QQ G$ comes with an anti-involution $r\mapsto r^\dagger$ which takes each basis 
element $e_g$ ($g\in G$) to the basis element  $e_{g^{-1}}$. This identifies $\QQ G$ with its opposite. 
Since $V_\chi$ is self-dual,  the involution leaves in the above decomposition each factor $ \End^{D_\chi}(V_\chi)$  
invariant and induces one in the skew-field
$D_\chi$:  the involution on $\End^{D_\chi}(V_\chi)$ is given by taking  the  $s_\chi$-adjoint:
$s_\chi (\sigma v, v')=s_\chi (v, \sigma^\dagger v')$. Since we have $D_\chi$ acting on $V_\chi$ on the right, 
this means that $s_\chi(v\lambda , v')=s_\chi(v, v'\lambda^\dagger)$.  (We note in passing that any other $G$-invariant inner  product 
$s'_\chi$ on $V_\chi$ is of the form $s_\chi(v\lambda , v')$ for some nonzero $\lambda$ with $\lambda^\dagger=\lambda$; the associated anti-involution of $D_\chi$ is then  a conjugate of $\dagger$.)
The center  of $D_\chi$, which we denote by $L_\chi$, is a number field and the fixed point set  of
$\dagger$ in $L_\chi$ is  a subfield  $K_\chi\subset L_\chi$ with $[L_\chi:K_\chi]\le 2$.

For a finitely generated  $\QQ G$-module $H$, we denote by  $H[\chi]$ the associated $\chi$-isogeny space $\hom_{\QQ G}(V_\chi, H)$. The right $D_\chi$-module structure on $V_\chi$ determines a  left $D_\chi$-module structure on $H[\chi]$ and the \emph{isotypical decomposition} of $H$ is the assertion that   the natural map
\[
\oplus_{\chi\in X(\QQ G)} V_\chi\otimes_{D_\chi} H[\chi]\to H, \quad v\otimes_{D_\chi} u\in V_\chi\otimes_{D_\chi} H[\chi]\mapsto u(v)
\]
is an  isomorphism of $\QQ G$-modules. So the \emph{$\chi$-isotypical subspace} of $H$, ie, the image $H_\chi$ of $V_\chi\otimes_{D_\chi} H[\chi]$ in $V$, has the structure of a $K_\chi$-vector space.

Assume now that $H$ comes equipped with a nondegenerate $G$-invariant symplectic  form $(a,b)\in H\times H\mapsto a\cdot b\in \QQ$. 
Then the isotypical  decomposition of $H$ is symplectic,  so that we 
also have a decomposition $\Sp(H)^G=\prod_{\chi\in X(\QQ G)} \Sp(H_\chi)^G$. Every factor $\Sp(H_\chi)^G$ can be understood with the help
of the skew-hermitian form
\[
(f, f')\in H[\chi]\times H[\chi]\mapsto \la f ,f' \ra_\chi\in D_\chi,
\]
which is characterized by the property that for all $v, v'\in V_\chi$,
\[
 f(v)\cdot f'(v')=s_\chi(v \la f, f'\ra_\chi,v')
\] 
 (\emph{skew-hermitian} means that
the form is $D_\chi$-linear in the first variable and $\la f', f\ra_\chi=-\la f,f'\ra_\chi^\dagger$). Indeed, for fixed 
$f$ and $f'$, the map $(v,v')\in V_\chi\times V\chi\mapsto f(v)\cdot f(v')\in \QQ$ is a  bilinear form on $V_\chi$. Since  $s_\chi: V_\chi\times V_\chi\to \QQ$ is nondegenerate,  there exists a unique  $\QQ$-linear endomorphism  $\sigma$ of $V_\chi$ such that
$f(v)\cdot f(v')=s_\chi(\sigma (v), v')$.  The $G$-invariance of both bilinear forms implies that 
$\sigma$ is $G$-equivariant, ie, is an element of  $\End_{\QQ G}(V_\chi)$. We prefer to regard it as an element of its opposite $D_\chi$, so that $f(v)\cdot f(v')=s_\chi(v\sigma, v')$. This $\sigma$ is evidently $\QQ$-linear in both $f$ and $f'$ and that is why we denote it  $\la f, f'\ra_\chi\in D_\chi$. It is then a little exercise to check that $\la \; , \; \ra_\chi$ is skew-hermitian.  This form is nondegenerate in an obvious sense. The group of automorphisms 
$H[\chi]$ that preserve this form is a generalized unitary group and therefore denoted $\U(H[\chi])$.

Any element of  $\Sp(H_\chi)^G$  acts via the isomorphism  $H_\chi\cong V_\chi\otimes_{D_\chi} H[\chi]$  as an element of the form $1_{V_\chi}\otimes u$ with $u\in\U(H[\chi])$ and this identifies $\Sp(H_\chi)^G$  with  $\U(H[\chi])$. The group $\U(H[\chi])$ is the group of $K_\chi$-points of a reductive algebraic  group defined over $K_\chi$, whereas $\Sp(H_\chi)^G$ is the group  of $\QQ$-points of an algebraic  group defined over $\QQ$. Indeed,   the latter is obtained from  the former by the restriction of  scalars $K_\chi|\QQ$.  

Theorem \ref{thm:main} then amounts to the assertion  that the image of $\diff^+(S)^G$ in the product of unitary groups 
$\prod_{\chi\in X(\QQ G)} \U(\Hl_1(S; \QQ)[\chi])$ is  arithmetic. We use this decomposition to prove the theorem, since  we first prove arithmeticity for a single factor. This  leads to a somewhat stronger result, for we show  that in setting of (i)  of our main theorem (so when $S\to S_G$ is trivial  over a genus 1 subsurface) the image of $\diff^+(S)^G$ in $\U(\Hl_1(S; \QQ)[\chi])$ is almost always arithmetic (see Remark \ref{rem:exceptions}). \\

Here is a brief description of the structure of the paper.  Of the four sections, only the last one is topological, but in order to put the  constructions given there to work, we need a considerable amount of algebra and that explains the nature of the preceding  sections. 

Section \ref{sect:special unitary} collects  useful (and essentially known)  algebraic proprieties of constructs that we encounter in the symplectic representation  theory of a finite group over $\QQ$.
So there is little or no claim  of originality here, although it was (for us) a bit of an effort  to extract this material  from the literature.
In Section \ref{sect:hyperbolic} we introduce and study what we  might regard as the basic symplectic module 
associated to an irreducible  $\QQ G$-module, where $G$ is a finite group. The main result is Proposition \ref{prop:arithm1}  which states an arithmeticity property and also lists the (few) cases for which this arithmetic group has real rank $\le 1$. This prepares us for stating and proving the arithmeticity criterion Theorem \ref{thm:transvectiongeneration} in Section \ref{sect:ac}, which  furnishes  the main algebraic input for Section  \ref{sect:liftable}. As mentioned,  this last section is essentially topological: we there construct sufficiently many $G$-equivariant mapping classes to ensure that we can apply the said theorem to obtain our main Theorem \ref{thm:main}.
\\

At various stages of this work I  benefited from correspondence with  colleagues on this material.  
These include many enlightening  exchanges with Marco Boggi,  who in 2019  provided me with a number of helpful comments on earlier drafts of the present paper. Tyakal Venkataramana explained to me in 2018 some of  the implications  of \cite{venky} and \cite{raghunathan}.  Justin Malestein helped me to  better understand  parts of \cite{GLLM}. I thank all of them. I am also very grateful to a number of referees for their sometimes extensive comments on an earlier version. The paper greatly improved as a result.

\section{Brief review of special unitary groups}\label{sect:special unitary}
\subsection*{The Albert classification}
In this subsection $D$  is a skew field of  finite dimension over $\QQ$ endowed with an anti-involution 
$\dagger$.  We assume that the involution  $\dagger$ is positive in the sense that 
$\lambda \in D\mapsto \tr_{D/\QQ}(\lambda \lambda^\dagger)$ is a positive definite form. 
We remark that this is so for the cases that matter 
here, for the given anti-involution on $\QQ G$ is evidently positive: for $r\in \QQ G$, the $\QQ$-trace of 
$r r^\dagger$ is $|G|$ times the coefficient of $e_1$ in  $r r^\dagger$ and hence is positive definite. 
The same is then true for its Wedderburn  factors  $\End^{D_\chi}(V_\chi)$  and their associated skew fields $D_\chi$.

We denote the center of $D$ by $L$  (so that is a number field) and by $D_+$ resp.\ $K$ the $\dagger$-invariant part of $D$ resp.\ $L$. 
Albert's  classification of such pairs $(D, \dagger)$ (see for example \cite{mumford:av}, Ch.~IV, Thm.~2) then tells us that $K$ is totally real,  so that $\RR\otimes_\QQ D=\prod_\sigma \RR\otimes_\sigma D$, where $\sigma$ runs over the distinct field embeddings $\sigma :K\hookrightarrow \RR$, and that there are essentially four cases:
\begin{enumerate}
\item[(I)] $D=L=K$ so that $\RR\otimes_\sigma D=\RR$ for each $\sigma$,
\item[(II)]  $L=K$ and for each $\sigma$ there exists an  isomorphism 
$\RR\otimes_\sigma D\cong \End_\RR(\RR^2)$ which sends $\dagger$ to taking the transpose (so $[D:L]=4$),
\item[(III)] $L=K$ and for each $\sigma$ there exists an  isomorphism 
$\RR\otimes_\sigma D\cong \KK$, where $\KK$ denotes the Hamilton quaternions,  which sends $\dagger$ to quaternion conjugation (so $[D:L]=4$), 
\item[(IV)]  $L$ is a purely imaginary extension of $K$ (in other words, $L$ is a CM field) and for each $\sigma$ there exists an  isomorphism 
$\RR\otimes_\sigma D\cong \End_\CC(\CC^d)$,   which takes $\RR\otimes_\sigma L$ to $\CC$ (so $[D:L]=d^2$) and sends $\dagger$ to taking the conjugate transpose.
\end{enumerate}

Let $M$ be a left $D$-module of finite rank. 
We write $M^\dagger$ for $M$ endowed with the structure of a right $D$-module via the rule $a\lambda :=\lambda^\dagger a$ ($a\in M$, $\lambda\in D$). So if $M'$ is another   left $D$-module, then $M^\dagger \otimes_{D} M'$ is defined. It is a $K$-vector space with the property that
$a\otimes_D \lambda a'=(\lambda^\dagger a)\otimes_D a'$ for all $\lambda\in D$, $a\in M$ and $a'\in M'$. In particular, we have in $M^\dagger \otimes_{D} M$ a $K$-linear involution defined by 
$(a\otimes_D b)'=b\otimes_D a$. We denote its fixed point set  $u(M)\subset M^\dagger \otimes_{D} M$. As a  $\QQ$-subspace of $M^\dagger \otimes_D M$ it is spanned by the symmetric tensors $a\otimes_D a$.

\subsection*{Isotropic transvections and Eichler transformations} Suppose given a  skew-hermitian form $(a,b )\in M\times M\mapsto \la a,b\ra\in  D$ on $M$.  
We denote its radical by $M_o$ so that the form  descends to a nondegenerate one on 
$\overline M:=M/M_o$. We define the associated  unitary  group $\U(M)$  as the group  of $D$-linear automorphisms of $M$ that preserve the form and act as the identity on $M_o$. It is the group of $K$-points of an algebraic group defined over $K$. If the form is nondegenerate ($M_o=0$), then  $\U(M)$  is what is called in \cite{HO}  (\S 5.2B) a \emph{classical unitary group}.

If $c\in M$ is \emph{isotropic} (meaning that $\la c,c\ra =0$), then we have the associated \emph{isotropic transvection}  
$T_c=T(c\otimes_D c)\in \U (M)$ defined by $x\in M\mapsto  x+ \la x,c\ra c$ (so $c\otimes_D c$ is here understood as an element of $M^\dagger\otimes_D M$).  It `generates' an abelian unipotent subgroup of $\U(M)$ defined by 
\[
\lambda\in D_+\mapsto T(c\otimes_D \lambda c)\in \U(M)\; , \;T(c\otimes_D \lambda c)(x)=x+\la x, \lambda c\ra c.
\] 
Isotropic transvections  are particular cases of  Eichler transformations. These are defined as follows. Let  $c\in M$ be isotropic,  $a\in M$  perpendicular to $c$ and $\lambda\in D$ such that $\lambda-\lambda^\dagger=\la a,a\ra$ (equivalently,  $\lambda-\frac{1}{2}\la a,a\ra  \in D_+$). Then the associated \emph{Eichler transformation} is
\[
E(c,a,\lambda): x\in M\mapsto x+\la x,a\ra c+\la x,c\ra a+ \la x,c\ra \lambda c\in M. 
\]
It is a $D$-linear transformation which preserves the form. When $\lambda=\frac{1}{2}\la a,a\ra $, we shall write $E(c,a)$ instead. Since $T(c\otimes_D c)=E(\frac{1}{2}c,c)$, isotropic transvections are Eichler transformations as asserted.

One checks  that each  Eichler transformation lies in $\U(M)$ as defined above. In fact, 
$t\in K\mapsto  E(tc,a,\lambda)=E(c,ta,t^2\lambda)$ is a closed one-parameter subgroup of $\U(M)$ whose infinitesimal generator is represented by $a\otimes_Dc +  c\otimes_Da\in u(M)$ (or rather by  its image in $u(M)/u(M_o)$, for if both $a$ and $c$ lie in $M_o$, then we get the identity).  By a general property of algebraic groups (\cite{Spr}, Cor.\ 2.2.7) such subgroups then generate a closed algebraic $K$-subgroup of  $\U(M)$. Following \cite{HO} we denote that group by $\EU(M)$.

We note the commutator identity
\begin{equation}\label{eqn:comm}
[E(c,a_1,\lambda_1), E(c,a_2,\lambda_2)]=T(c\otimes \lambda c), \text{  with  } \lambda=\la a_1, a_2\ra+\la a_1, a_2\ra^\dagger.
\end{equation}
It follows that if we fix $c$, but let $a$ and $\lambda$ vary (subject to the conditions above, so with $a\in c^\perp$), then the $E(c, a, \lambda)$ generate a unipotent group that appears as an extension
of the vector group $c^\perp/(M_o+Dc)$ by the abelian subgroup of $\U(M)$ defined by the $T(c\otimes_D \lambda c)$.  

The group $\EU(M)$ is already generated by the isotropic transvections: When $\dagger$ is nontrivial this follows from (6.3.1) of  \cite{HO}. The remaining case is  the one we labelled (I): this is when $D=L=K$ and $\U(M)$ is a symplectic group over $K$, but then there is no issue for then every  $a\in M$ is isotropic, and then $E(c,a)=T_{a+c}T_a^{-1}T_c^{-1}$. 

\subsection*{Unipotent radical and Levi quotient} 
If the form is nondegenerate (ie, $M_o=\{0\}$), then  $\EU(M)$ is a  $K$-form of a classical semisimple algebraic group and hence has finite center. To be precise, it is a group of symplectic type in the cases (I) and (II), of orthogonal type in case (III) and of special linear type in case  (IV).

If $M_o$ is possibly nonzero, then per convention the elements of $\U(M)$ act trivially on $M_o$.
The natural map $\U(M)\to \U(\overline M)$ is  evidently onto. Its kernel consists of the transformations that act trivially on both $M_o$ and $M/M_o$ and is therefore the  unipotent radical $\ru(\U(M))$ of $\U(M)$ (recall that $\U(\overline M)$ is reductive): we have an exact sequence
\[
1\to \ru(\U(M))\to \U(M)\to \U(\overline M)\to 1
\]
The  elements of $\ru(\U(M))$ are the Eichler transformations $E(c,a)$ with $c\in M_o$ and $a\in M$ arbitrary. In this case $E(c,a)$ only depends on the image of $c\otimes_\ZZ a$ in $M_o^\dagger\otimes_D \overline M)$, so that the resulting map
\[
M_0^\dagger\otimes_D \overline M\to \ru (\U(M))
\]
is  an isomorphism. So $\ru(\U(M))$ is a vector group over $K$ (ie, a $K$-vector space regarded as the group of $K$-points of a $K$-algebraic group of additive type). Since $\EU(\overline M)$ is a normal semisimple subgroup of $\U(M)$, it has the same unipotent radical: $\ru (\EU(M))=\ru (\U(M))$.

\subsection*{The relation between $\EU (M)$ and $\U(M)$} 
Since in what follows the notion of real rank of an algebraic group shows up,  let us begin with reviewing this concept briefly.  

Let $\Gs$ be  a reductive  algebraic group. Suppose first that $\Gs$ is defined over $\RR$. Then  the real rank $\rank_\RR(\Gs)$ of $\Gs$ is by definition 
the dimension of a Cartan subgroup of $\Gs$ defined over $\RR$. For example, if $\Gs$ is the 
orthogonal group of a nondegenerate quadratic form over $\RR$, then its real rank is the 
Witt index of this form: the dimension of a maximal isotropic subspace defined over $\RR$. 
If  $\Gs$ is defined over a number field $k$, then
we restrict scalars \`a la Weil so that $\Res_{k|\QQ}\Gs$ is a group defined over $\QQ$. We then regard $\Res_{k|\QQ}\Gs$ (by base change) as a group over $\RR$ and define the real rank of $\Gs$  to be the real rank of the latter. Concretely, if  $\sigma_1, \dots, \sigma_r$ are the real embeddings of $k$ in $\RR$ and 
 $\tau_{1}, \overline\tau_{1}, \dots, \tau_s, \overline\tau_s$ are the remaining distinct 
 (complex) embeddings (they come in complex conjugate pairs), then the definition comes down to 
 \[
\textstyle  \rank_\RR(\Gs)=\sum_{i=1}^r \rank_\RR(\Gs_{\sigma_i}) +\sum_{i=1}^{s} \rank_\CC(\Gs_{\tau_i}) 
\]
(this is also the sum over all the  archimedean valuations of $k$, taking as  general term the real rank of the corresponding completion of $\Gs(k)$). The Dirichlet unit theorem  often gives lower bounds for the rank.
For example,  if the skew field $D$  is  as in the Albert classification,  
the group of units $D^\times $  is a reductive group defined over $K$ and its group of real points and   its real rank are then as follows: putting $e:=[K:\QQ]$, then 
\begin{enumerate}
\item[(I)] $\Res_{K|\QQ}D^\times(\RR)$ is open in $(\RR^\times)^e$;  the  real rank of $D^\times$ is $e$,
\item[(II)]  $\Res_{K|\QQ}D^\times(\RR)$ is open in $\GL_2(\RR)^e$; the  real rank of $D^\times$ is $2e$,
\item[(III)] $\Res_{K|\QQ}D^\times(\RR)$ is open in $(\KK^\times)^e$; the  real rank of $D^\times$ is $e$,
\item[(IV)] $\Res_{K|\QQ}D^\times(\RR)$ is open in $\GL_d(\CC)^e$; the  real rank of $D^\times$  is 
$de$.
\end{enumerate}
So  $D^\times $ has real rank $\ge 2$, unless  $D$ equals $\QQ$ (I) or is 
a definite quaternion algebra with center $\QQ$ (III) or is an imaginary quadratic extension of $\QQ$.
\\

It is clear that $\EU(M)$ is a normal subgroup of $\U(M)$. We already noticed that it is closed in $\U(M)$ and hence the quotient  $\U(M)/\EU (M)$ is also an algebraic group. Note however that if $\overline M$ has no nonzero isotropic vectors, then $\EU(\overline M)$ is trivial.  We mention for future reference  a consequence of a theorem of 
G.E.~Wall \cite{wall}:

\begin{lemma}\label{lemma:norm}
Suppose that  $M$ is nondegenerate and contains a nonzero isotropic vector. Then  $\U(M)/\EU (M)$ is anisotropic (all its real forms are compact).  So $\EU(M)$ and $\U (M)$ have the  same real rank and any arithmetic subgroup of $\U(M)$ will have finite image in $\U(M)/\EU (M)$.
\end{lemma}
\begin{proof}
Thm.\ 1 of  \cite{wall} identifies $\U(M)/\EU (M)$ as a quotient of $D^\times$ by a normal subgroup which contains $D^\times\cap D_+$.  It is easy to check that in all four cases (I)-(IV) in the Albert classification such a quotient must be anisotropic. 
\end{proof}

Wall's result is more specific and tells us that $\U(M)/\EU (M)$ is often an anisotropic torus. But this need not be so when $\dim_D M=2$.  

\begin{lemma}\label{lemma:simple}
If $M$ is nondegenerate isotropic, then the  $K$-algebraic group $\EU(M)$ is almost simple 
(by which we mean that $\EU(M)$ is perfect and every proper normal subgroup is contained in its center) unless $D=K$ 
and  $M\cong K^4$ is endowed with a nondegenerate symmetric form which admits an 
isotropic plane defined over $K$. 
\end{lemma}
\begin{proof}
This follows  from \cite{HO}, Thm.\ 6.3.16 combined with  Thm.\  6.3.15. 
\end{proof}

The excepted case is genuine,  for in that case $M\cong M_1\otimes_K M_2$ as modules endowed with  $K$-forms, where  $M_i$ is a 2-dimensional $K$-vector space endowed with a nondegenerate symplectic form. The resulting map  $\SL(M_1)\times\SL(M_2)\to\GL(M)$ has image $\EU(M)\cong \Orth(M)$ and its  kernel has $(-1,-1)$ as its unique nonidentity element.

\begin{remark}\label{rem:}
The  \emph{reduced norm} is the homomorphism $N: \U(M)\to L^\times$ characterized by the following property: if $T\in \U(M)$, then for some (or equivalently, every) real embedding $\sigma :K\hookrightarrow \RR$, the $D$-linear  $T$ induces a linear transformation of the $\RR\otimes_\sigma L$-vector space $\RR\otimes_K M$
whose determinant is $1\otimes_\sigma N(T)$.
The kernel of $N$, usually denoted $\SU(M)$,  contains $\EU(M)$ and is often equal to it. But in our context this group does not show up in a natural manner.
\end{remark}

\section{The hyperbolic module attached to a finite group}\label{sect:hyperbolic}

\subsection*{A hermitian extension} Our discussion of symplectic $\QQ G$-modules also applies to orthogonal $\QQ G$-modules.
One such module is $\QQ G$ itself (regarded as a left module):  it comes indeed with $G$-invariant
pairing $\QQ G\times \QQ G\to \QQ$, the \emph{trace form}, which assigns to the pair $(r_1,r_2)$ the trace of $r_1r_2^\dagger$ considered as an endomorphism of $\QQ G$ as a $\QQ$-vector space (this is simply $|G|$ times the coefficient of $e_1$). This pairing is symmetric and nondegenerate.

The Wedderburn decomposition  $\QQ G\cong\prod_\chi\End^{D_\chi}(V_\chi)$ is also the isotypical decomposition, for the $K_\chi$-linear map
\[
V_\chi\otimes_{D_\chi} \hom^{D_\chi}(V_\chi, D_\chi)\to \End^{D_\chi}(V_\chi), \quad 
v_1\otimes_{D_\chi} f : v\in V_\chi\mapsto v_1f(v)
\]
which assigns to $v_1\otimes_{D_\chi} f$ the endomorphism $v\in V_\chi\mapsto v_1f(v)$ is well-defined and is an isomorphism of left  $\QQ G$-modules (and also as right $\QQ G$-modules). This also shows that 
$\QQ G[\chi]\cong  \hom^{D_\chi}(V_\chi, D_\chi)$ as a right $D_\chi$-module. 

We claim that $\hom^{D_\chi}(V_\chi, D_\chi)\cong V_\chi^\dagger$ as left $D_\chi$-modules.  This is based on a hermitian extension of $s_\chi$ to $V_\chi^\dagger$: if we follow the same recipe as in the introduction for a symplectic representation, then we find that there is  a  hermitian form $h_\chi:V_\chi^\dagger\times V_\chi^\dagger\to D_\chi$  characterized by the property that for all $v, v'\in V_\chi$ and $f, f'\in V_\chi^\dagger$,
\[
s_\chi(v, f')s_\chi(v', f)=s_\chi (vh_\chi (f, f'), v').
\]
This formula implies  that $h_\chi$ is $G$-invariant (we let $G$ act on the right of $V_\chi^\dagger$).
By taking $v=f=f'$, we also see that $h_\chi (f, f)=s_\chi (f,f)$, so that   $h_\chi$ is a hermitian extension of $s_\chi$. For every $v\in V_\chi^\dagger$, the expression $h_\chi(v,\; -)$ yields an element of $\hom^{D_\chi}(V_\chi, D_\chi)$ and defines the stated isomorphism.

\subsection*{Isotropic transvections}
Here and in the rest of this paper, we write  $R$ for the integral group ring  $\ZZ G$.
Let $M$ be a finitely generated (left) $R$-module, free over $\ZZ$ and let $(a,b)\in M\times M\mapsto a\cdot b\in \ZZ$ be a nondegenerate (but not necessarily unimodular) $G$-equivariant symplectic form. We extend this in the standard manner to a form 
\[
\textstyle (a,b)\in M\times M\mapsto \la a, b\ra :=\sum_{g\in G} (g^{-1}a\cdot b) e_g=\sum_{g\in G} (a\cdot gb) e_g\in R.
\]
This form is skew-hermitian: it is $R$-linear in the first variable and $\la a, b\ra=-\la b,a\ra^\dagger$. A $\ZZ$-linear automorphism of $M$ is $G$-equivariant and preserves the symplectic form if and only if it is an $R$-module automorphism which preserves this skew-hermitian form. We denote the group of such automorphisms by $\U(M)$. 

Let $R_+$ stand for the fixed point set of  $\dagger$ in $R$; this is an additive subgroup of $R$.  If  $a\in M$ is \emph{$R$-isotropic}  in the sense that $\la a, a\ra=0$,  then for every $r\in R_+$  the isotropic transvection
\begin{equation}\label{eqn:Tformula}
T_a(r): x\in M\mapsto x +\la x,a\ra ra\in M
\end{equation}
lies in $U(M)$ and $r\in R_+\mapsto T_a(r)$ is a homomorphism  from (the additively written) $R_+$ to (the multiplicatively written) $U(M)$. Since $T_a(r)$ only depends on $a\otimes ra\in M^\dagger\otimes_R M$, we also denote this transformation by  $T(a\otimes ra)$.

\subsection*{The basic hyperbolic module}
Let $A$ be a (not necessarily commutative) unital ring with unit endowed with an anti-involution $\dagger$.
The \emph{basic hyperbolic $A$-module}  $\Hcal^2(A)$ is the free left $A$-module of rank $2$  (whose  generators we  denote $\eb$ and $\fb$) endowed with the skew-hermitian form defined by $\la \eb, \eb\ra=\la \fb, \fb\ra=0$ and $\la \eb, \fb\ra=1$. It can be regarded as the $A$-form of the standard symplectic module $\ZZ^2$. In vector notation:
\begin{equation}\label{eqn:form}
\Big\la\binom{a'}{a''} , \binom{b'}{b''}\Big\ra =a'b''{}^\dagger-a''b'{}^\dagger.
\end{equation}
It is unimodular in the sense that $a\in \Hcal^2(A)\mapsto \la -,a\ra\in \hom_A(\Hcal^2(A), A)$ is an antilinear isomorphism. 
We will write  $\U_2(A)$ resp.\  $\EU_2(A)$ for $\U(\Hcal^2(A))$ resp.\ $\EU(\Hcal^2(A))$.
The latter contains $\SL_2(\ZZ)$ in an obvious manner. Let $A_+\subset  A$ be the set of $\dagger$-invariant elements. One verifies that  $T_\eb: x\in \Hcal^2(A)\mapsto  x+ \la x,\eb\ra \eb$ and the similarly defined   $T_\fb$ have the matrix form
\[  
T_\eb(a)=(\begin{smallmatrix}
1 & -\rho_a\\
0 &  1\\
\end{smallmatrix}) 
\text{  resp.\  }
T_\fb(a)=(\begin{smallmatrix}
1 & 1\\
\rho_a &  1\\
\end{smallmatrix}), 
\]
where $\rho_a$ stands for right multiplication with $a$.

\subsection*{The group $\G(G)$}
We take $A=R(=\ZZ G)$. The elements of the form $r+r^\dagger$ with $r\in R$, make up a  subgroup $R_{++}$ of $R_+$ such that $R_+/R_{++}$ is a finite dimensional  $\FF_2$-vector space.   Let  $\G_+(G)$ resp.\  $\G_-(G)$ denote  the subgroup of  $\EU_2(R)$ generated by $T_\eb(R_{++})$ resp.\ $T_\fb(R_{++})$ and let  $\Gamma_o(G)\subset \EU_2(R)$ stand for the subgroup 
 generated by $\G_+(G)$ and $\SL_2(\ZZ)$.  Since  $\G_-(G)$ is a $\SL_2(\ZZ)$-conjugate of $\G_+(G)$, the group
$\Gamma_o(G)$ contains $T_\fb(R_{++})$.
The right (inverse) action of $G$ on $\Hcal^2(R)$ defines an embedding of $G$ in $\U_2(R)$. One checks that
\[
\rho_gT_\eb(r)\rho_{g^\dagger}= T_\eb (gr g^{-1})
\]
and so $G$ normalizes $\Gamma_o(G)$. We put $\G(G):=\Gamma_o(G).G$.

\subsection*{Arithmetic nature of $\Gamma(G)$}
The notion of a basic hyperbolic module generalizes in a straightforward manner to $\Hcal^2(V_\chi^\dagger)$, the skew-hermitian form being given by
\begin{equation}\label{eqn:form2}
\Big\la\binom{v'}{v''} , \binom{w'}{w''}\Big\ra =h_\chi(v',w'')-h_\chi(v'',w').
\end{equation}
So we have defined $\U_2(V_\chi^\dagger)$; it is the group of $K_\chi$-points of a reductive  algebraic group defined over $K_\chi$. If we write an element of $\U_2(V_\chi^\dagger)$ in block form
$(\begin{smallmatrix}
A & B\\
C&  D\\
\end{smallmatrix})$, with $A,B,C,D$ in $\End_{D_\chi}(V_\chi^\dagger)$, then the subgroup defined by $C=0$ is a  parabolic subgroup.
Its unipotent radical is given by requiring that in addition $A$ and $D$ are the identity. The corresponding subgroup is  then the vector  group  $(\begin{smallmatrix}
1 & B\\
0&  1\\
\end{smallmatrix})$ for which $B$ is hermitian relative to $h_\chi$. In other words, $h_\chi$ identifies $B$ with an element of
$u(V_\chi^\dagger)$ (that is, a symmetric element of $V_\chi\otimes_{D_\chi}V_\chi^\dagger$) and hence defines an isotropic transvection.  In particular, this is also the unipotent radical of the corresponding subgroup of $\EU_2(V_\chi^\dagger)$.  An opposite parabolic subgroup is defined by $B=0$ and has a similar description of its unipotent radical. Let us denote these unipotent radicals $\cR_+(\U_2(V_\chi^\dagger))$ resp.\  $\cR_-(\U_2(V_\chi^\dagger))$.

We run into this  when we consider the isotypical decomposition of $\Hcal^2 (\QQ G)$.
The isogeny space $\Hcal^2(\QQ G)[\chi]=\hom_{\QQ G} (V_\chi, \Hcal^2(\QQ G))$ is then a left $D_\chi$-module that is naturally identified with $\Hcal^2(V_\chi^\dagger)$. This gives rise to a decomposition
\[
\U_2(\QQ G)=\prod_{\chi\in X(\QQ G)} \U_2(V_\chi^\dagger).
\]
The image of   $\G_\pm(G)$ in $\U_2(V_\chi^\dagger)$ clearly lands in $\cR_\pm(\U_2(V_\chi^\dagger))$.
Since $\End_{D_\chi}(V_\chi^\dagger)=\End^{D_\chi}(V_\chi)$ is a Wedderburn factor of $\QQ G$, it follows
that the image of $R$ in $\End_{D_\chi}(V_\chi^\dagger)$ is an order (a lattice that is also a unital algebra).
This is compatible with the anti-involutions and hence the image of $R_{++}$ in $\End_{D_\chi}(V_\chi^\dagger)$ is a lattice in the subspace of  hermitian matrices. So the image of $\G_\pm(G)$ in
$\U_2(V_\chi^\dagger)$ is a lattice in $\cR_\pm(\U_2(V_\chi^\dagger))$.

It is clear  that $\Gamma_o(G)$  maps to $\EU_2(V_\chi^\dagger)$. 
\\

\begin{proposition}\label{prop:arithm1}
The image of  $\Gamma(G)$ in $\U_2(V_\chi^\dagger)$  is an arithmetic subgroup. The real rank of $\U_2(V_\chi^\dagger)$ is $\ge 2$ unless $\dim_{D_\chi}V_\chi=1$. In that last case,  where we can  assume that $V^\dagger_\chi=D_\chi$ with $G$ acting on the right and mapping to its group of units, one of the following  holds:
\begin{enumerate}
\item[(i)] $D_\chi=\QQ$ and  $G$  maps to $\mu_2$, 
\item[(iia)]  $D_\chi$ is the Gaussian field $\QQ(\sqrt{-1})$  and $G$ maps onto $\mu_4$, or
\item[(iib)] $D_\chi$ is the Eisenstein field $\QQ(\sqrt{-3})$  and  $G$ maps onto  $\mu_3$ or $\mu_6$, or
\item[(iiia)] $D_\chi\cong \QQ +\QQ\mathbf{i}+\QQ\mathbf{j}+\QQ\mathbf{k}$ and  $G$ maps onto
its group of units (a binary tetrahedral group of order $24$) or onto the quaternion  group of order $8$, or  
\item[(iiib)] $D_\chi\cong \QQ +\QQ\sqrt{3}\mathbf{i}+\QQ\mathbf{j}+\QQ\sqrt{3}\mathbf{k}$ and $G$ maps onto the binary dihedral group of order $12$.
\end{enumerate}
In all  cases, $\Gamma(G)$ acts  $\QQ$-irreducibly in $\Hcal^2(V_\chi^\dagger)$.
\end{proposition}

\begin{remark}\label{rem:EMS}It is well-known that the quaternion group appearing in \ref{prop:arithm1}-(iiia) is realized as  the Galois group of a torus ramified at four points (the covering surface has genus 3). This example is like a Swiss army knife for illustrating (and  refuting) statements in complex dynamics,  which  is why  that community refers to it as the \emph{eierlegende  Wollmilchsau}. We do not know whether its appearance here is just a coincidence.
\end{remark}

For the proof we need:

\begin{theorem}[Raghunathan \cite{raghunathan}, Venkataramana \cite{venky}]\label{thm:generation} 
Let $\cG$ be an almost simple,  simply connected $\QQ$-algebraic group of real rank $\ge 2$. Let $\cR_-$ and $\cR_+$ be $\QQ$-subgroups that contain the unipotent radicals of opposite $\QQ$-parabolic subgroups of 
$\cG$. Then for any pair of lattices $\Gamma_+\subset \cR_+(\QQ)$,  $\Gamma_-\subset \cR_-(\QQ)$, the subgroup of  $\cG(\QQ)$ generated by their union $\Gamma_+\cup\Gamma_-$ is a congruence subgroup of $\cG(\QQ)$.
\end{theorem}
\begin{proof}
We first  prove the arithmeticity property of $\Gamma_o(G)$ in $\EU_2(V_\chi^\dagger)$.
Let us first observe that the group $\EU_2(V_\chi^\dagger)$ is almost-simple by 
Lemma \ref{lemma:simple}. Indeed, this can only fail if 
$D_\chi=K_\chi$ and  the form  is symmetric (with $\dim_{K_\chi}V_\chi^\dagger=2$) and this  is clearly not the case.

If the real rank of  $\EU_2(V_\chi^\dagger)$ is $\ge 2$, then the  theorem of Raghunathan-Venkataramana applies and we conclude that $\Gamma_o(\chi)$ is an arithmetic subgroup of $\EU_2(V_\chi^\dagger)$. It is then also easy to see that  $\Gamma(G)$ acts  $\QQ$-irreducibly in $\Hcal^2(V_\chi^\dagger)$.

Since the real rank of $\EU_2(V_\chi^\dagger)$ is  $\ge \dim_{D_\chi}V_\chi^\dagger$, 
it remains  to treat the case when  $\dim_{D_\chi}V_\chi^\dagger =1$. In other words,  
we can assume that $V_\chi^\dagger=D_\chi$.
The real rank of $\EU_2(D_\chi)$ is then still $\ge 2$ most of the time. As 
we saw above, the exceptions are the cases for which  $K_\chi=\QQ$ and  $D_\chi$ is  either 
$\QQ$, a definite quaternion algebra over $\QQ$ with center $\QQ$ or an imaginary quadratic extension of $\QQ$. Since $D_\chi$ is  also an irreducible  $\QQ G$-module, we have a
homomorphism $\rho: G\to D_\chi^\times$  whose image contains a $\QQ$-basis of $D_\chi$. 
In particular, $D$ is generated over $\QQ$ by its units.
So  if $D_\chi$ is an imaginary quadratic extension of $\QQ$, then $D_\chi$ is either $\QQ$, 
the Gaussian field or the Eisenstein field. In the definite quaternion case, $D_\chi^\times(\RR)$ 
is the group of unit quaternions  and hence  $\rho(G)$ is one of the  subgroups classified by Klein: 
this group must be  binary tetrahedral, 
binary octahedral,  binary icosahedral  of  binary dihedral group (of order $4n$). In these cases
$\rho(\QQ G)\cap \RR$ equals resp.\  $\QQ$, $\QQ(\sqrt{2})$,  $\QQ(\sqrt{5})$, $\QQ(\cos(\pi/n))$.
Since we want this intersection to be $\QQ$,  only the two groups listed have that property.

Note that in each of these exceptional cases, $D_{\chi, +}=K_\chi=\QQ$. The isotropic subspaces in $\Hcal^2(D_\chi)$ are defined over $\QQ$ and hence  
$\EU_2(D_\chi)\cong \SL_2(\QQ)$. The group  $\Gamma_o(\chi)$ is then a copy of $\SL_2(\ZZ)$. So $\Gamma_o(G)$  is arithmetic in $\EU_2(V_\chi^\dagger)$.
In view of Lemma \ref{lemma:norm} this also implies the arithmeticity of $\G(G)$ in $\U_2(V_\chi^\dagger)$.
The actions of $G$ (on the right) and  $\SL_2(\ZZ)$ (on the left) on $\Hcal^2(V^\dagger_\chi)$  
commute and  make $\Hcal^2(V^\dagger_\chi)$  an exterior  tensor product of  irreducible  
$\QQ$-representations: it is the  right $\QQ G$-module $V^\dagger_\chi$ tensored with the 
tautological representation of  $\SL_2(\ZZ)$ on $\QQ^2$ (which is absolutely irreducible). Hence 
$\Hcal^2(V^\dagger_\chi)$ is irreducible as a representation of $\SL_2(\ZZ)\times G$. 
This implies that $\Hcal^2(V^\dagger_\chi)$ is irreducible  as a $\G (G)$-module.
 \end{proof}

\section{An arithmeticity criterion}\label{sect:ac} 

In this section we fix a rational character $\chi\in X(\QQ G)$. We therefore suppress the subscript $\chi$ and write $D$ for $D_\chi$ and $V$ for $V_\chi$. 

Proposition  \ref{prop:arithm1} tells us that $\G (G)\subset \U_2(V^\dagger)$  is an arithmetic subgroup  which acts $\QQ$-irreducibly on  $\Hcal^2(V^\dagger)$ and  that with  a few exceptions, the group $\U_2(V^\dagger)$ is of real rank $\ge 2$.

\subsection*{Eichler transformations revisited}  Let $M$ be (left) $D$-module  of finite rank endowed with a  nondegenerate skew-hermitian  form $\la\; ,\, \ra :M\times M\to D$.  Given a $D$-submodule $N\subset M$, we denote by $\U_M(N)$  the subgroup of the  group of transformations  that act trivially on $N^\perp$. This group preserves $N$ and acts trivially on its radical $N_o=N\cap N^\perp$. Hence `restriction to $N$' defines  homomorphism $\U_M(N)\to \U(N)$.  This homomorphism  is easily shown to be onto. Its kernel consists of the unitary transformations that act trivially on $N+N^\perp$ and  one checks that this is  the image of $u(N_o)$ under $T$. We  saw that the homomorphism $\U(N)\to \U(\overline N)$ is also onto and we identified its kernel with the vector group $N_o^\dagger\otimes \overline N$. So the Levi quotient of 
$\U_M(N)$ is $\U(\overline N)$  and  its  unipotent radical  $ \ru (\U_M(N))$ is an extension of vector groups: 
\[
1\to u(N_o)\xrightarrow{T} \ru (\U_M(N))\to N_o^\dagger\otimes_D \overline N\to 0.
\]
As is clear from the Equation \eqref{eqn:comm}, this  extension is usually nontrivial. 
In case  $N_o$ is spanned by a single element $c$, then we can write this sequence  as
\begin{equation}\label{eqn:extension}
0\to c\otimes D_+c\xrightarrow{T}\ru (\U_M(N))\to  c\otimes \overline N\to 0.
\end{equation}
Any element of $\ru (\U_M(N))$ is an  Eichler transformation $E(c, a, \lambda)$ whose image 
in $c\otimes \overline N$ is $c\otimes \overline a$ (where $\overline a\in \overline N$ is the image of $a$). We will often use the following lemma.

\begin{lemma}\label{lemma:latticeprop}
Let $\G\subset \U_M(N)$ be a discrete subgroup whose image in $\U(\overline N)$ is arithmetic  and which acts $\QQ$-irreducibly in $\overline N$.
If  $\G\cap \ru (\U_M(N))$ contains an Eichler transformation $E(c, a, \lambda)$ with $a\in N\ssm D_+c$, then 
$\G\cap \U_M(N)$ is arithmetic in $\U_M(N)$.
\end{lemma}
\begin{proof} We are given that in the exact sequence of algebraic groups
\[
1\to \ru (\U_M(N))\to \U_M(N)\to \U(\overline N)\to 1
\]
the image $\overline \G$ of $\G$ in $\U(\overline N)$ is arithmetic. Hence for $\G$ to be arithmetic, it suffices that  $\G\cap \ru (\EU_M(N))$ be a lattice. For this we turn to the exact sequence \eqref{eqn:extension}. The Eichler transformation $E(c, a, \lambda)$ has image 
$c\otimes \overline a$ in $c\otimes \overline N$ and this image is nonzero by assumption.
The image of the $\G$-conjugacy class of $E(c, a, \lambda)$ in $c\otimes \overline N$ is equal to
$c\otimes \overline\G \overline a$.  Since our assumptions also imply that $\overline \G$ acts $\QQ$-irreducibly in $\overline N$, it follows that the  image of $\G\cap \ru (\U_M(N))$ in $c\otimes \overline N$ is a lattice in $c\otimes \overline N$. 

Next observe that if $E(c, a_1, \lambda_1)$ and  $E(c, a_2, \lambda_2)$ lie in $\G\cap \ru (\U_M(N))$, then so  does  their commutator, which by the identity \eqref{eqn:comm} is $T(c\otimes \lambda c)$ with $\lambda=\la a_1,a_2\ra +\la a_1,a_2\ra^\dagger$. Since  the 
$\la a_1,a_2\ra$ generate a lattice in $D$, it follows that the $\lambda$ generate a lattice in $D_+$.
In other words, the preimage of $\G\cap \ru (\U_M(N))$ in $c\otimes D_+c$ is also a lattice. Hence 
$\G\cap \ru (\U_M(N))$ is a lattice.
\end{proof}

\subsection*{Hyperbolic submodules}
If   $j:\Hcal^2(V^\dagger)\hookrightarrow M$ is an embedding of hermitian $D$-modules, then $M$ is the  orthogonal direct sum of the image of $j$  and its perp (for $\Hcal^2(V^\dagger)$ is nondegenerate),  so that  $j$ gives rise to an injective homomorphism of groups $j_*: \U_2(V^\dagger)\hookrightarrow \U(M)$. 
Let us refer to such an embedding as a \emph{$V^\dagger$-hyperbolic summand} in $M$.

The following  criterion for arithmeticity will be central to our argument.

\begin{theorem}\label{thm:transvectiongeneration}
Let  $M$ be a nondegenerate skew-hermitian $D$-module of finite rank and 
\[
a: V^\dagger\hookrightarrow M, \quad \{b: V^\dagger\hookrightarrow M\}_{b\in \Bs}
\]
a  collection of $D$-linear embeddings 
(with $\Bs$ finite, nonempty) whose images span $M$ over $D$ and are such that for each $b\in \Bs$ the pair  $(a,b)$ defines a hyperbolic summand of $M$. In case $\dim_D V^\dagger=1$ \emph{and} $\dim_DM>2$, assume in addition that there exist $b_1, b_2\in \Bs$  for which $b_1(V^\dagger)$ and $b_2(V^\dagger)$ are perpendicular.  

Then the subgroup $\G$ of $\U(M)$ generated by  $\{(a,b)_*\G(G)\}_{b\in \Bs}$ is an arithmetic subgroup of  $\U(M)$ which acts $\QQ$-irreducibly in $M$.
\end{theorem}

The proof will be by  induction on $\dim_DM$.  As may be inferred from the statement of the theorem, the case when $\dim_DV^\dagger=1$ is a bit more delicate. Indeed,  the first  induction step then requires special  care and so  we do that case first. Once we have dealt with it,  we indicate how to modify the arguments in order  to obtain a proof of the unrestricted  version of Theorem \ref{thm:transvectiongeneration}.

Let us say that a $D$-subspace $N\subset M$ is \emph{$\G$-arithmetic} if $\Gamma\cap\U_M(N)$ is arithmetic
  in $\U_M(N)$ and acts $\QQ$-irreducibly in $\overline N$ (the last property is a consequence of the first if the real rank of $\U_M(N)$ is $\ge 2$).

\subsection*{The case $\dim_DV^\dagger=1$} We then identify $V^\dagger$ with $D$ and $a$ and each $b\in\Bs$ with the image of $1\in D$ under these embeddings so that  $\la a, b\ra =1$ for all $b\in \Bs$. 
Note that $\{a\}\cup\Bs\subset M$ consists of isotropic elements  and generates $M$ over $D$.  We  write
$\G(a,b)$ for the image of $\G(G)$ under $(a,b)$, so that  $\G$ is generated by $\{\G(a,b)\}_{b\in \Bs}$. As any $b\in\Bs$ lies $\G(a,b)a$, it follows that $\{a\}\cup \Bs\subset \G a$. 

By Proposition \ref{prop:arithm1}, $Da +Db$ is $\G$-arithmetic for every $b\in \Bs$.  We therefore assume that $M$ is not of the form $Da +Db$.  So there exist $b_1, b_2\in \Bs$ with $b_2\notin Da +Db_1$ such that $\la b_1, b_2\ra=0$. 


\begin{lemma}\label{lemma:step1}
Put $N:=Da+Db_1$. Then $N':=N+Db_2$ is  $\G$-arithmetic.
\end{lemma}
\begin{proof}  
 We verify that the  assumptions of the Lemma \ref{lemma:latticeprop} are satisfied by $\G\cap \U_M(N')$.
 It is clear that the radical of $N'$ is spanned by $c:=b_2-b_1$ so that $\overline N'$ is  the isomorphic image of $N$. We know that  $N$ is $\G$-arithmetic and  so  $\G$ has arithmetic image in $\U (\overline N')$. Since $T_{b_1}$ and $T_{b_2}$ lie in $\G$, so does $ T_{b_2}T_{b_1}^{-1}$. We check that 
 \[
T_{b_2}T_{b_1}^{-1}(x)=x-\la x, b_1\ra +\la x, b_1+c\ra (b_1+c)=E(c, b_1,1)(x).
\]
So the image of $\G\cap \ru( \U_M(N'))$ in $c\otimes \overline N'$ contains $c\otimes b_1$. Now apply Lemma \ref{lemma:latticeprop}.
\end{proof}

From this point onward the argument will be inductive. The union of  Lemmas \ref{lemma:step2} and \ref{lemma:step3} will establish the theorem  in case $\dim_DV^\dagger=1$. 

\begin{lemma}\label{lemma:step2}
Let $N\subsetneq M$ be  a $D$-subspace which contains $a, b_1, b_2$
and whose the radical is of $D$-dimension one.  If $N$ is $\G$-arithmetic, then there exists a $b\in \Bs$  such that $N':=N+Db$ is nondegenerate and $\G$-arithmetic, and  the real rank of  $\U(N')$ is $\ge 2$.
\end{lemma}
\begin{proof}
Let $c\in N$ span the  radical of $N$.
Since $M$ is nondegenerate and $D$-spanned by $\{a\}\cup\Bs$,  there must exist a $b\in \Bs$ such that $c$ is not in the radical of 
$N':=N+Db$. Then it is easily seen that $N'$ is nondegenerate so that $\U_M(N')\cong \U(N')$.  
In case $N=Da +Db_1+Db_2$ (where we can take $c=b_2-b_1$) one checks that $N$ is the perpendicular sum of two copies of  $\Hcal^2(D)$. Otherwise, $N'$ contains such a sum. This implies that 
$\U(N')$ has real rank $\ge 2$.  

The $\U(N')$-stabilizer $\U(N')_c$ of $c$ is equal to $\U_M(N)$ and hence contains
$\G\cap \U(N')_c$ as an arithmetic subgroup. Observe that $c':=T_b(c)=c+\la c,b\ra b$ is another isotropic element 
with $\la c',c\ra=\la c,b\ra \la b,c\ra\not=0$ and so $Dc+Dc'\cong \Hcal^2(D)$.   
The two $\U(N')$-stabilizers of $Dc$ and $Dc'$ are opposite parabolic subgroups of $\U (N')$ whose unipotent radicals are contained in  $\U(N')_{c}$ resp.\  $\U(N')_{c'}$. Since $\U(N')_{c'}$ is a $\G$-conjugate of $\U(N')_c$, it follows that
$\G\cap \U(N')_{c'}$ is an arithmetic subgroup of $\U(N')_{c'}$. We have thus  satisfied  the hypotheses of Theorem \ref{thm:generation} and we conclude that $\G\cap \U(N')$ is arithmetic in $\U(N')$. The fact that 
$\G\cap \U(N')$ acts $\QQ$-irreducibly in $N'$ follows from the fact  that 
$\G\cap \U(N)$ has this property in $N$,  for  the $\G\cap \U(N')$-translates of $N$ span $N'$ over $\QQ$, but do not decompose $N'$.
\end{proof}

\begin{lemma}\label{lemma:step3}
Let $N\subsetneq M$ be  a proper nondegenerate $D$-subspace of dimension $\ge 4$ and    contain $a, b_1, b_2$. If $N$ is $\G$-arithmetic, then so is $N':=N+Db'$ for every $b'\in \Bs\ssm N$.
\end{lemma}

\begin{proof}[Proof in case $N'$ is degenerate] 
We verify that the  assumptions of the Lemma \ref{lemma:latticeprop} are satisfied by $\G\cap \U_M(N')$.
The radical of $N'$ is necessary spanned by an element of the form  $c:=b'-b$ where  
$b\in N$ is characterized by the property that $\la x, b\ra=\la x, b'\ra$ for all $x\in N$. So $N$ maps then isomorphically onto $N'/Dc=\overline N'$. In particular, the natural map $\U(N)\to \U(\overline N')$ is an isomorphism and hence $\G\cap \U_M(N')$  
maps onto an arithmetic subgroup of $\U(\overline N')$. 

Let $n$ be a positive  integer such that $T_{b}^n\in \G$.  Then
\[
T_{b'}^nT^{-n}_{b}(x)=x-n\la x,b\ra b+n\la x,b'\ra b'=x+ n\la x,b\ra c+n\la x,c\ra b +n\la x,c\ra c
=E(c,nb,n)(x)
\]
So $E(c,nb,n)\in \G\cap\ru(\U_M(N'))$ and this element has image $c\otimes nb$ in $c\otimes \overline N'$. It then follows from  Lemma  \ref{lemma:latticeprop} that $N'$ is $\G$-arithmetic.
\end{proof}

\begin{proof}[Proof in case $N'$ is nondegenerate] 
Then $N'_1:=N'\cap b'{}^\perp$ is degenerate with radical spanned by $b'$. We first prove that 
$N'_1$ is $\G$-arithmetic by  verifying the  assumptions of the Lemma \ref{lemma:latticeprop}.
The  subspace $N_1:=N\cap b'{}^\perp$ supplements $b'$ in $N'_1$. It  is therefore  nondegenerate and maps isomorphically onto $\overline N'_1=N'_1/Db'$. This enables us to regard  $\U(N_1)$ as  a  subgroup of $\U_M(N'_1)$ that acts trivially on both $b'$ and its orthogonal projection in $N$.

Since $\G\cap \U_M(N)$ is arithmetic in $\U_M(N)$, its subgroup $\G\cap \U(N_1)$ is  arithmetic in 
$\U(N_1)$ and  has  arithmetic image in $\U(\overline N'_1)$. We show that $E(b',c')\in \G$ for some $c'\in N'_1\ssm Db'$. Then  Lemma  \ref{lemma:latticeprop} will imply that $N'_1$ is $\G$-arithmetic.

For this  we recall that $c:=b_2-b_1$ is perpendicular to
$b_1$ and has nonzero image in $D b_1^\perp/Db_1$.
Let  $\gamma\in \G(a,b')\G(a,b_1)\subset \G(N')$ 
take $b_1$ to $b'$. Since $c$ is perpendicular to $b_1$,   $c':=\gamma(c)\in N'$  is perpendicular to $\gamma (b_1)=b'$ (so lies in $N'_1$) and
has nonzero image $[c']$ in $\overline N'_1\cong N_1$. 
Since $E(b',c')=E(c', b')^{-1}$ is a  $\G$-conjugate of $E(c, b_1)^{-1}$,  it lies in $\G$.

Now is $\U(N')_{b'}=\U(N'\cap b'{}^\perp)$ and so $\G\cap \U(N')_{b'}$ is arithmetic in
$\U(N')_{b'}$. As $a\in \G b'$, the same is true for $\G\cap \U (N')_{a}$. Since $a$ and $b'$ span a copy of $\Hcal^2(D)$, their $\U(N')$-stabilizers contain the unipotent radicals of  opposite parabolic subgroups  of $\U (N')$. The real rank of $\U (N')$ is $\ge 2$, so that Theorem \ref{thm:generation} applies and tells us that $N'$ is $\G$-arithmetic.
\end{proof}
 
\subsection*{The case when $\dim_DV^\dagger>1$}
The same scheme  for the proof works when $\dim_DV^\dagger>1$. The difference is that we deal with  larger hyperbolic packets, to wit, the  images of hyperbolic embeddings $(a,b):\Hcal^2(V^\dagger)\hookrightarrow M$. The essential difference is that we start off in a better position, since  we begin with a $V^\dagger$-hyperbolic embedding $j:\Hcal^2(V^\dagger)\hookrightarrow M$ and we already know that  its image  $N$ is $\G$-arithmetic and that $\U_M(N)\cong \U(N)$ has real rank $\ge 2$.

\section{Finding liftable mapping classes}\label{sect:liftable}
In this section the $G$-covering  $S\to S_G$ is as in the introduction and $A\subset S_G^\circ$ is a nonempty closed $1$-submanifold
such that the covering is trivial over $A$ and is connected over $S\ssm A$. We also choose a connected component $\alpha$ of $A$, so that  $A':=A\ssm \alpha$ might be empty.
We orient $\alpha$ and regard it as  the oriented image of an embedding of the circle in $S_G^\circ$. We will see that this gives rise to enough copies of $\G(G)$ in the representation 
of the $G$-equivariant mapping classes as to satisfy the hypotheses of our arithmeticity criterion Theorem  \ref{thm:transvectiongeneration}.
\\

We denote by $S_G(\alpha)$ the singular surface obtained from $S_G$ by contracting 
$\alpha$ to a point (that we will denote by $\infty$). Its topological normalization is a closed connected surface, denoted $\widehat S_G(\alpha)$, whose genus is one less than that of $S_G$. The surface $\widehat S_G(\alpha)$
comes with two points over $\infty$ and the orientation of $\alpha$ enables us to tell them apart: we let  $p_-$ be  `to the left' of $\alpha$  and $p_+$ 
is `to the right' of $\alpha$. If  we regard $A'$ also as a submanifold of $\widehat S_G^\circ(\alpha)$, then the surjection $\widehat S_G^\circ(\alpha)\to S^\circ_G(\alpha)$ defines a map from the set 
$\Pi(\widehat S_G^\circ(\alpha)\ssm A';p_-, p_+)$ of path homotopy classes in $\widehat S_G^\circ(\alpha)\ssm A'$ from $p_-$ to $p_+$ to the fundamental group $\pi_1(S^\circ_G(\alpha), \infty)$. This map is injective. 
We do the same  (in a $G$-equivariant manner) for the preimage of  $S_G\ssm\alpha$ in $S$ and thus get $G$-covers $S(\alpha)\to 
S_G(\alpha)$, $\widehat S^\circ(\alpha)\to \widehat S_G(\alpha)$  and $G$-orbits $P_\infty \subset S(\alpha)^\circ$, $P_\pm\subset \widehat S^\circ(\alpha)$, so that  we end up with the diagram below (in which the vertical maps are $G$-coverings):
\begin{center}
\begin{tikzcd}[row sep=scriptsize, column sep=scriptsize]
\pi^{-1}\alpha\arrow[dr, hook]\arrow[dd, two heads] \arrow[rr, two heads] & & P_\infty\arrow[dr,hook]\arrow[dd, two heads] & &P_-\cup P_+\arrow[dr, hook]\arrow[dd, two heads]\arrow[ll,two heads]&\\
& S\arrow[rr, crossing over, two heads]  && S(\alpha)&&\widehat S(\alpha)\arrow[ll, crossing over, two heads]\arrow[dd]\\
\alpha\arrow[dr, hook]\arrow[rr, crossing over, two heads] & & \{\infty\}\arrow[dr,hook] & &\{p_-, p_+\}\arrow[dr, hook]\arrow[ll,two heads]&\\
&S_G\arrow[from=uu, crossing over, two heads]\arrow[rr,two heads] && S_G(\alpha)
\arrow[from=uu, crossing over, two heads] && \widehat S_G(\alpha)\arrow[ll,two heads]\\
\end{tikzcd}
\end{center}
This construction comes with $G$-equivariant bijections $P_-\cong P_\infty\cong P_+$. Our assumption  on the covering  $S\to S_G$ amounts to the following two properties: (i)  $\widehat S(\alpha)$ is connected and stays so if we remove the preimage of $A'$ and (ii) the three  $G$-orbits $P_\infty$, $P_\pm$ are regular. So the choice of a point in $P_\infty$ (which is equivalent to the choice of a lift $\tilde\alpha$ of $\alpha$) identifies these $G$-sets with $G$ (on which $G$ acts by left translation). In particular, we thus  identify  $\Iso_G(P_-,P_+)\cong\aut_G(P_\infty)$ with $G$ (where $g\in G$  acts on $G$ by right translation over $g^{-1}$).

\begin{figure}
\centerline{\includegraphics[scale=0.250]{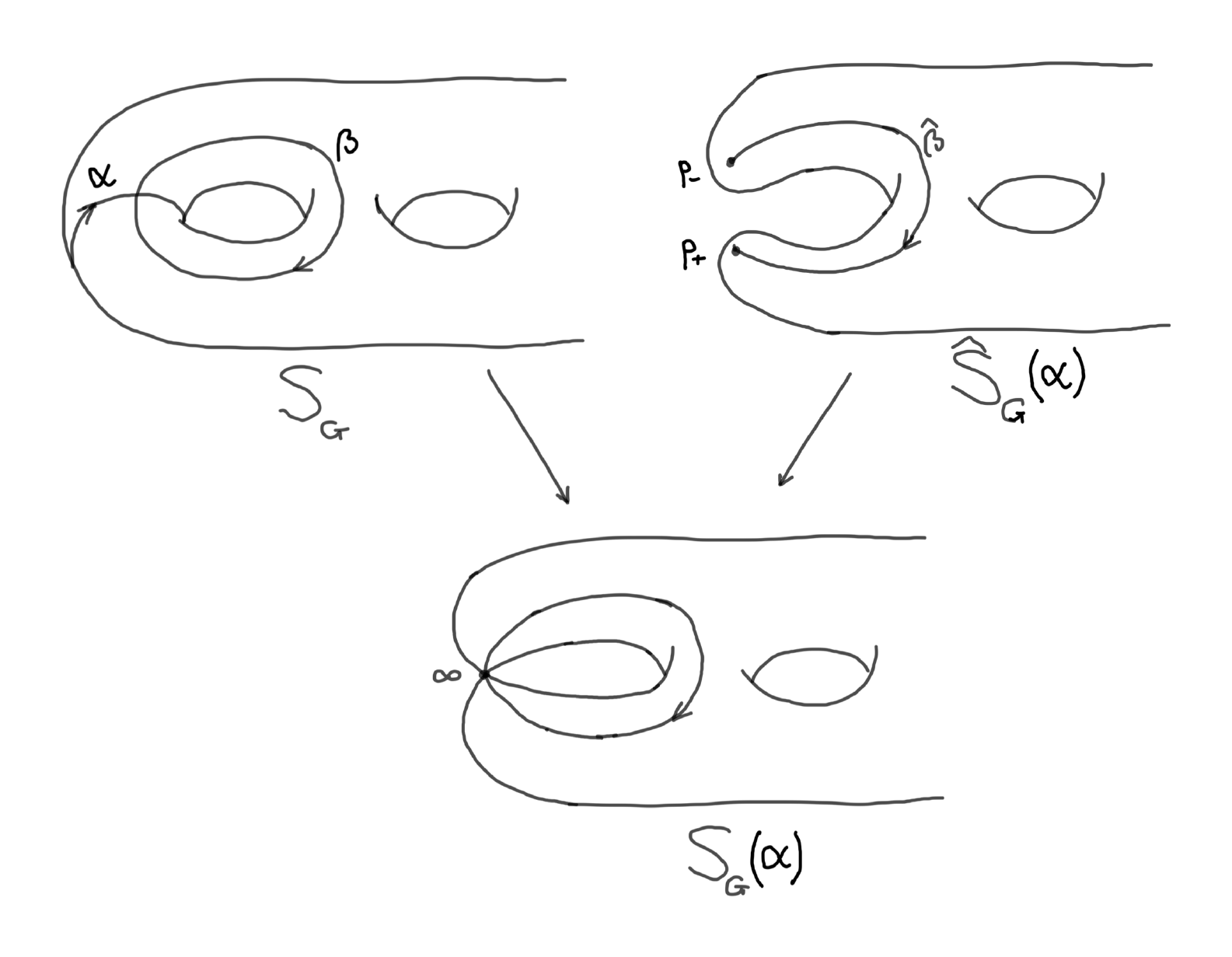}}
\caption{\footnotesize  The surface $S_G$, its quotient $S_G(\alpha)$ and the normalization $\widehat S_G(\alpha)$.}
\label{fig:surfacepic}
\end{figure}

For any  path in $\widehat S_G^\circ(\alpha)\ssm A'$ from $p_-$ to $p_+$, the $G$ covering  is trivial over it, so that we have  an associated  $G$-bijection $P_-\cong P_+$. Since the $G$-covering over $\widehat S_G^\circ(\alpha)\ssm A'$ is connected, the resulting  map  
\[
\Pi(\widehat S_G^\circ(\alpha)\ssm A' ;p_-,p_+)\to \Iso_G(P_-,P_+)\cong \aut_G(P_\infty)
\] 
is onto by standard covering theory. We will say that
an element of $\Pi(\widehat S_G^\circ(\alpha)\ssm A'; p_-, p_+)$ \emph{is $G$-trivial} if its image in $\aut_G(P_\infty)$ under the above map is the identity. Such elements make up a coset for
the kernel of the natural homomorphism $\pi_1(\widehat S_G^\circ(\alpha)\ssm A',p_-)\to G$.

\begin{lemma}\label{lemma:Ghit}
Every element  of $\Pi(\widehat S_G^\circ(\alpha)\ssm A'; p_-, p_+)$ is representable by some arc ($=$ embedded unit interval) in $\widehat S_G^\circ(\alpha)\ssm  A' $ from $p_-$ to $p_+$.
We can arrange that this arc lifts to an embedding of the circle $\RR/\ZZ$ in $S^\circ_G$ which meets $\alpha$ in a single point with intersection number $1$. In particular, every element of
$\aut_G(P)$ is realized by the monodromy along an embedded circle $\beta$ which does not meet $A'$ and meets $\alpha$ in one point only and does so transversally with intersection number  $1$. 
\end{lemma}
\begin{proof}
We first represent the homotopy class by an immersion of the unit interval
with only transversal self-intersections. That number of self-intersections is finite and if this number is positive, we lower it  by moving  the last point of self-intersection towards $p_+$  and then slide the path over $p_+$. By iterating this procedure we obtain a representative which is an embedding. It is clear that we can make this arc lift to an embedded circle. The second assertion then follows.
\end{proof}

We choose a lift $\tilde\alpha$ of $\alpha$ and write $a\in \Hl_1(S)$ for its homology class. Since $Ra$ is an an isotropic sublattice of $\Hl_1(S)$, we have that $\la a,a\ra=0$. We let $I\subset \Hl_1(S)$ be the homology supported by the preimage of $A'$:
this is a free $R$-submodule (where as before, $R=\ZZ G$) with a generator for every connected component of $A'$. It is clear that $Ra +I$ is isotropic.
We saw that the lift $\tilde\alpha$ identifies $\aut_G(P)$ with the group $G$ with $G$ acting on itself by right translations.  
Lemma \ref{lemma:Ghit} above shows that all such elements are  obtained from a loop of the type described there. From that lemma we also derive:

\begin{corollary}\label{cor:rep}
Let $b\in I^\perp$  be such that $a \cdot b=1$ and  $ga \cdot b=0$  for $g\in G\ssm\{1\}$. Then some  $b'\in b+ R a$ can be represented by a lift of an embedded circle $\beta$ in $S^\circ_G\ssm A' $ which meets $\alpha$ transversally at a unique point (and for which necessarily  $\alpha\cdot\beta=1$). In particular, $b'$ is $R$-isotropic and the $R$-linear map defined by
\[
\Hcal^2(R)\to I^\perp, \quad \eb\mapsto a,\; \fb\mapsto b' 
\]  
is an embedding of skew-hermitian $R$-modules whose orthogonal complement supplements its image:  we obtain a basic hyperbolic summand of the $R$-module $\Hl_1(S)$.

\end{corollary}
\begin{proof}
It is not difficult to see that the homology class $b$ is representable by a map from the circle 
to $S^\circ$ which meets the preimage of $\alpha$ exactly once 
(hence in a point of $\tilde\alpha$) and does not meet the preimage of $A'$. 
We apply Lemma \ref{lemma:Ghit} to its image in $S^\circ_G\ssm A'$ or rather to  
the resulting arc in $\widehat S_G^\circ(\alpha)\ssm A'$ which connects $p_-$ with $p_+$: 
this then produces an embedding  $\beta$ of the circle in $S^\circ_G\ssm A'$ which meets 
$\alpha$ only once and with intersection number $1$  over which the $G$-covering is trivial. 
The  lift $\tilde\beta$ of $\beta$ which meets $\tilde\alpha$ defines a homology  class $b'$ 
which differs from by $b$ by a class supported by the preimage of $\alpha$, 
that is, an element of $Ra$. 

The proof of the last paragraph is straightforward.
\end{proof}

Let us call an ordered  pair   $a,b$ in $\Hl_1(S)$ \emph{$R$-hyperbolic} if $\la a, a\ra=\la b, b\ra =0$ and $\la a,b\ra=1$. Such a pair defines a  
basic hyperbolic summand $Ra+Rb \subset \Hl_1(S)$ and gives rise to an embedding of $\G(G)$ in $\Sp(\Hl_1(S))^G$. We shall denote the latter's image by $\G (a,b)$ and the image of $\Gamma_o(G)$ by $\Gamma_o(a,b)$. 
We write $\Bs_a$  for the set of $b\in \Hl_1(S)$ for which the pair $(a,b)$ is $R$-hyperbolic and $\Gamma_o(a)$ resp.\  $\G (a)$ for  the subgroup of $\Sp(\Hl_1(S))^G$  generated by its subgroups $\Gamma_o(a,b)$ resp.\ $\G(a,b)$ with $b\in\Bs_a$.
\\

Fix a base point $p$ on $\alpha$ and write  $\tilde p$ for its preimage in $\tilde\alpha$. 

Let $h\in G$ and regard $h$ as an element of $\aut_G(P)$.  Let $\beta$  be as in Lemma \ref{lemma:Ghit}, which we may (and will) assume to meet $\alpha$ in 
$p$ such that the lift $\tilde\beta$ of $\beta$ (as an arc) which begins in $\tilde p$ ends in $h\tilde p$.
Let   $S^{\beta}_G$ be a closed regular neighborhood
of $\alpha\cup\beta$ in $S^\circ_G\ssm A'$. This is a compact genus one surface whose boundary $\partial S^{\beta}_G$ is connected. The homotopy class of this boundary (with  its natural orientation) is in the free homotopy class of the commutator
$[\beta]^{-1} [\alpha]^{-1} [\beta][\alpha]$ (we write path composition functorially, so the order of travel is read from right to left). This commutator has trivial image in $G$ (since $[\alpha]$ has), and so the $G$-covering $S\to S_G$ is trivial over $\partial S^{\beta}_G$. The preimage of $\partial S^{\beta}_G$ in $S$ is the boundary of the 
preimage  $S^{\beta}$ of $S^{\beta}_G$ and the Dehn twist along $\partial S^{\beta}_G$ 
lifts in a $G$-equivariant manner to a multi-Dehn twist $D^\beta$  along that boundary. The following  lemma generalizes one of the constructions given in \cite{L} for  the  case when $G$ is cyclic (in that paper they are  depicted as Figs.\ 2 and 3). 

\begin{lemma}\label{lemma:multidehn} 
The multi-Dehn twist $D^\beta$ acts on $\Hl_1(S)$ as $T_a(2-e_h-e^\dagger_h)$ (where we use   formula \eqref{eqn:Tformula}, noting that $2-e_h-e^\dagger_h\in R_+$).
\end{lemma}
\begin{proof}
The lift of the commutator $[\beta]^{-1} [\alpha]^{-1} [\beta][\alpha]$ that passes through $\tilde p$ first traverses the embedded circle $\tilde\alpha$, then traverses $\tilde\beta$,  then traverses the circle $h\tilde\alpha$ in the opposite direction and then returns via the inverse of $\tilde\beta$ to $\tilde p$. 
So the homology class of this lift of the commutator (and hence of the corresponding lift of $\partial S^{\beta}_G$) is $a-ha$. If we replace $\tilde p$ by $g\tilde p$ with $g\in G$, then this replaces $a$ by $ga$ and $h$ by $ghg^{-1}$, so that  the corresponding class is $ga-gha=g(1-h)a$.  By a standard formula, the resulting action on $\Hl_1(S)$  is given by 
\begin{multline*}
\textstyle D^\beta_*(x)=x+\sum_{g\in G} (x\cdot g(1-h)a) g(1-h) a=\\
\textstyle =x+\sum_{g\in G} \Big((x\cdot ga) ga -( x\cdot gh a) ga - (x\cdot ga) gha +(x\cdot gha) gha\Big)=\\
=x+2\la x, a\ra a -\la x ,ha\ra a -\la x ,a\ra h a= T_a(2-e_h-e^\dagger_h)(x).\qedhere
\end{multline*}
\end{proof}

\begin{proposition}\label{prop:imageinclusion}
Let $b\in \Hl_1(S)$ be such that $(a,b)$ is an $R$-hyperbolic pair. Then the image  of $\diff^+(S)^G\to \Sp(\Hl_1(S))^G$ contains $\G (a,b)$. 
\end{proposition}
\begin{proof}
We first show this for $\G_0 (a,b)$. Let $\beta$ be as in Corollary \ref{cor:rep} (and thus represent an element of $b+Ra$). 
The diffeomorphisms of $S_G$ with support in the interior of $S^\beta_G$  have as their image in the mapping class group
of $S_G$ a  centrally extended copy of $\SL(2, \ZZ)$ with the central subgroup generated by the Dehn twist along the boundary of $S^\beta_G$.  This Dehn twist acts trivially on $\Hl_1(S^\beta_G)$.  These diffeomorphisms lift to diffeomorphisms of $S$ with support in $S^\beta$  with  the central subgroup acting trivially on $\Hl_1(S)$. 
We thus obtain in the image of $\diff^+(S)^G\to \Sp(\Hl_1(S))^G$ a copy of $\SL_2(\ZZ)$. 

The multi-Dehn twist  associated to $\alpha$ acts on  $\Hl_1(S)$ as $x\mapsto x+\sum_{g\in G} (x\cdot ga) ga =x+\la x, a\ra a=T_a(1)(x)$. By lemma \ref{lemma:multidehn} the image of $\diff^+(S)^G\to \Sp(\Hl_1(S))^G$ also contains the transvections $T_a(2-e_h -e_h^\dagger)$ for all $h\in G$. Hence 
that image contains all of $T_a(R_{++})$. This proves that the image  of $\diff^+(S)^G\to \Sp(\Hl_1(S))^G$ contains $\G_o(a,b)$. 

So it remains to show that the image of $G$ in $\G(a,b)$  is realized by $\diff^+(S)^G$. For this we use mapping classes of \emph{push type}.
Consider the smooth surface  $S_G/S_G^\beta$  that is  obtained as a quotient of $S_G$ by contracting  $S_G^\beta$  to a point (that we shall call $q$).  If we do the same for the connected components of $S^\beta$ in $G$ we get a $G$-cover $S/\!\!/S^\beta\to S_G/S_G^\beta$ fitting  in the commutative diagram
\[
\begin{CD}
S  @>>> S/\!\!/S^\beta\\
@V{G}VV @V{G}VV\\
S_G @>>> S_G/S^\beta_G
\end{CD}
\]
The covering on the right does not branch over $q$ and so its preimage $Q$ in $S/\!\!/S^\beta$ is a regular $G$-orbit. For every $g \in G$, there is a closed 
loop $\gamma$ of $S_G/S^\beta_G$ based at $q$ which avoids branch points and induces in $Q$ right multiplication by $g$. The corresponding point-pushing map on $S_G/S^\beta_G$ (chosen to fix branch points) lifts to a $G$-equivariant diffeomorphism  $\varphi$ of
$S/\!\!/S^\beta$ that extends this permutation. 
\begin{figure}
\centerline{\includegraphics[scale=0.200]{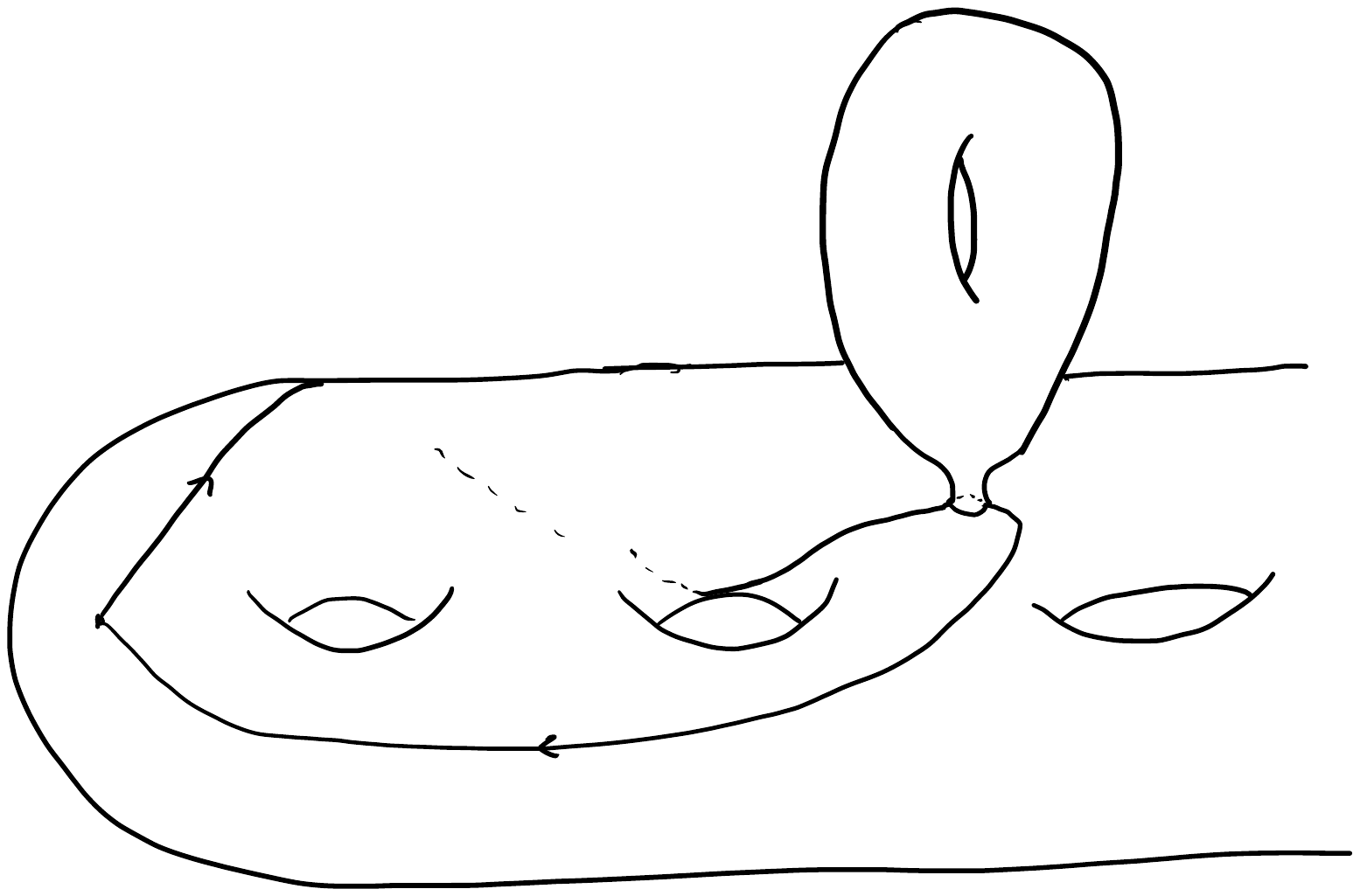}}
\caption{\footnotesize  Point pushing (or rather, `small circle pushing') the  genus one surface $S^\beta$ along $S_G/S^\beta_G$.}
\label{fig:onthemove}
\end{figure}

Such a  point-pushing map  is isotopic to the identity on $S_G/S^\beta_G$ and hence the same is true for its lift $\varphi$. In particular,  $\varphi$ acts trivially on  $\Hl_1(S/\!\!/S^\beta)$. It is not difficult to see that  the point pushing map  and its lift $\phi$ can be `lifted' to $S$ and $S_G$ by `small circle pushing'. Since  $\Hl_1(S\ssm S^\beta, \partial(S\ssm S^\beta))\to  \Hl^1(S/\!\!/S^\beta)$ is an isomorphism, the action on 
$\Hl_1(S\ssm S^\beta, \partial(S\ssm S^\beta))$ will be trivial. Clearly the components of $S_\beta$ will be permuted according to the right action of $g$ and thus realizes $g\in\G(a,b)$ in the image of $\diff^+(S)^G$.
\end{proof}

Part (ii) of the corollary below  establishes the Putman-Wieland property of Theorem \ref{thm:main}.

\begin{corollary}[Hyperbolic generation]\label{cor:hyperbolicgen}
The following properties hold:
\begin{enumerate}
\item[(i)] the subset $\{a\}\cup \Bs_a$ of $\Hl_1(S; \QQ)$ spans the latter over $\QQ G$,
\item[(ii)] the subgroup $\Gamma_o(a)$ of $\Sp  (\Hl_1(S))$ generated by the subgroups $\Gamma_o(a,b)$ with $b\in \Bs_a$ (and hence $\diff^+(S)^G$) has no nonzero finite orbit in $\Hl_1(S)$.
\end{enumerate}
\end{corollary}
\begin{proof}
Let $c\in \Hl_1(S; \QQ)$ be perpendicular to $\{a\}\cup \Bs_a$. We prove that $c$ is then perpendicular to 
every  $x\in \Hl_1(S)$; since the intersection form is nondegenerate, this will imply that  $c=0$ and hence that  the $\ZZ G$-submodule of $\Hl_1(S)$ generated $\{a\}\cup \Bs_a$ is of finite index.  To this end, let $b\in \Bs_a$. Then $x':=(1+\la x,a\ra)b+ x$ has the property that $\la a,x'\ra=1$. By Corollary \ref{cor:rep} is   $(a,x'')$ is a hyperbolic pair for some  $x''\in x'+ R a$, so that  $0=\la x'',c\ra=\la x,c\ra$.

For (ii) it suffices to show that for every finite index subgroup $\G\subset \Gamma_o(a)$,
the fixed part $\Hl_1(S)^\G$ is trivial.
Note that  $\Hl_1(S)^{\Gamma_o(a,b)}$ is the perp of $Ra+Rb$ in $\Hl_1(S)$ with respect to the intersection pairing.
The $\Gamma_o(a,b)$-invariant part of  $\Hl_1(S)$ is not changed if we replace $\Gamma_o(a,b)$ by the finite index subgroup $\G\cap \Gamma_o(a,b)$ and hence  $\Hl_1(S)^{\G}$ is  perpendicular  to $Ra+Rb$. As this is true for all $b\in \Bs_a$ and $\{a\}\cup \Bs_a$ generates $\Hl_1(S)$ as an $R$-module, it follows that  $\Hl_1(S)^{\G}$ must be trivial.
\end{proof}

We can now finish  the proof of our main theorem.
\begin{proof}[Proof of the main theorem \ref{thm:main}]
Let us denote the image of $\diff^+(M)^G$  in $\Sp(\Hl_1(S; \QQ))^G$ by $\G$  and  the image of the latter
in the factor $\U(\Hl_1(S)[\chi])$ of  $\Sp(\Hl_1(S; \QQ))^G=\prod_\chi \U(\Hl_1(S)[\chi])$ by
$\G_\chi$. By combining Corollary \ref{cor:hyperbolicgen} with Theorem \ref{thm:transvectiongeneration}, 
we see that under the assumptions of (ii), $\G_\chi$ is an arithmetic subgroup of $\U(\Hl_1(S)[\chi])$.

Note that  $\G^\chi:=\G\cap \EU(\Hl_1(S)[\chi])$ is a normal subgroup of $\G_\chi$. 
It remains to see that
$\G^\chi$ is of finite index in $\G_\chi$.  Since $\EU(\Hl_1(S)[\chi])$ is almost simple  and of real rank $\ge 2$,   
it follows from a general result of Margulis  (\cite{margulis}, Assertion (A), Ch.~VIII) that this is the case unless $\G^\chi$ meets 
$\EU(\Hl_1(S)[\chi])$ in the center. But Proposition \ref{prop:imageinclusion} shows that $\G^\chi$ contains a subgroup isomorphic to 
 $\Gamma_o(\chi)$ and so this last possibility does not occur. 
\end{proof}

\begin{remark}\label{rem:exceptions}
This argument shows that if we are in the setting of (i) (triviality of the cover over a genus one subsurface of $S_G$), then $\G_\chi$ is an arithmetic subgroup of $\U(\Hl_1(S)[\chi])$, unless  $V_\chi^\dagger\cong D_\chi$ and the image of $G$ in $V_\chi$ is of the type given in Proposition \ref{prop:arithm1}. Denote that image by $G_\chi$ and let $G^\chi$ stand for the kernel of $G\to G_\chi$. Since $\Hl_1(S)[\chi]$  already arises on the $G_\chi$-cover $S_{G^\chi}\to S_G$ in the sense that  $\Hl_1(S)[\chi]=\Hl_1(S_{G^\chi})[\chi]$, we may  for the arithmeticity  question just as well focus  on this  intermediate cover. 

In the cases (i) and (ii) of \ref{prop:arithm1}, so when $D_\chi$ equals $\QQ$, the Gaussian field or the Eisenstein field, then $G_\chi$ is a group of roots of unity and hence cyclic.  When the genus of $S_G$ is at least $3$, we can always find a closed subsurface of genus 2  over which the covering $S_{G^\chi}\to S_G$ is trivial and so  $\G_\chi$ is then arithmetic.  In the remaining cases, $G_\chi$ is a particular kind of Kleinian group. It might well be that  a $G_\chi$-cover is then also trivial  over the complement of a genus two subsurface of the quotient surface when the genus of the latter is  $\ge 3$. If true, then it would follow that $\G$ would always be arithmetic if the genus of $S_G$ is at least $3$ and the covering  is trivial over a genus one subsurface of $S_G$.
\end{remark}

\end{document}